\documentclass[oneside,english]{amsart}
\usepackage[T1]{fontenc}
\usepackage[latin9]{inputenc}
\usepackage{amsthm}
\usepackage{amstext}
\usepackage{amssymb}
\usepackage{esint}
\usepackage{amscd}
\usepackage{color}
\makeatletter
\numberwithin{equation}{section}
\numberwithin{figure}{section}
\theoremstyle{plain}
\newtheorem{thm}{\protect\theoremname}[section]

\theoremstyle{plain}
\theoremstyle{definition}

\theoremstyle{plain}
\newtheorem{lem}[thm]{\protect\lemmaname}

\theoremstyle{plain}

\theoremstyle{plain}
\makeatother

\usepackage{babel}
\providecommand{\definitionname}{Definition}
\providecommand{\lemmaname}{Lemma}
\providecommand{\theoremname}{Theorem}
\providecommand{\corollaryname}{Corollary}
\providecommand{\remarkname}{Remark}
\providecommand{\propositionname}{Proposition}
	
\DeclareMathOperator{\loc}{loc}

\DeclareMathOperator{\diam}{diam}

\begin{document}

\title[Estimates of variational eigenvalues on metric measure spaces]{Estimates of variational eigenvalues on metric measure spaces}

\author{Prashanta Garain and Alexander Ukhlov}
\begin{abstract}
In the article, we study variational eigenvalues on doubling metric measure spaces. We prove existence of minimizers of variational Neumann $(p,q)$-eigenvalues on metric measure spaces and on this base we obtain estimates of Neumann eigenvalues.
\end{abstract}
\maketitle

\section{Introduction}

The theory of Neumann eigenvalues of nonlinear elliptic operators has a long history (see, for example, \cite{A98,AB93,EP15,GU17} and references therein). The Neumann $(p,q)$-eigenvalues in Euclidean domains $\Omega\subset\mathbb R^n$ given by 
\begin{equation}\label{maineqn}
-\Delta_p u:=-\text{div}(|\nabla u|^{p-2}\nabla u)=\lambda\|u\|_{L^q(\Omega)}^{p-q}|u|^{q-2}u\text{ in }\Omega,\quad\frac{\partial u}{\partial\nu}=0\text{ on }\partial \Omega,
\end{equation}
were considered in \cite{GPU24}. In \cite{GPU24}, estimates of the Neumann $(p,q)$-eigenvalues were obtained using variational characterizations. In the present article, we prove the existence of minimizers for the variational Neumann $(p,q)$-eigenvalue problem on metric measure spaces and, based on this, derive estimates for Neumann eigenvalues.

Let $(X, \rho, \mu)$ denote a metric measure space, where $\rho$ is a metric and $\mu$ is a doubling Borel regular measure. Let $\Omega\subset X$ be a domain. We define the first nontrivial variational Neumann $(p,q)$-eigenvalue, $1<q,p<\infty$, as 
$$
\lambda_{p,q}(\Omega)\\=\inf \left\{\frac{\|\nabla_{min} u\|_{L^p(\Omega)}^p}{\|u\|_{L^q(\Omega)}^p} :
u \in N^{1,p}(\Omega) \setminus \{0\}, \int_{\Omega} |u|^{q-2}u~d\mu(x)=0 \right\},
$$
where $\nabla_{min} u$ is a minimal upper gradient of $u\in N^{1,p}(\Omega)$.

The theory of Sobolev spaces on metric measure spaces arises in the 1990s in connection with the pointwise behavior of Sobolev functions \cite{H93, V96}. In a series of subsequent works \cite{GT01,H03,HK00,Sh00}, the foundations of the theory of Sobolev spaces were established, see \cite{HKST}.

Sobolev spaces provide the natural setting for weak solutions of partial differential equations \cite{S36}. Although it is not immediately clear what the proper counterpart of the equation \eqref{maineqn} should be in a general metric measure space, a variational approach is available, allowing eigenfunctions to be defined as minimizers of the corresponding energy integral. The regularity and existence of minimizers for the $p$-Dirichlet integral have been extensively studied in recent years (see, for example, \cite{GK,KS01,MS18} and references therein).

In the present article, we prove the existence of minimizers for the integral variational quantity corresponding to the Neumann eigenvalue problem on metric measure spaces. Based on this result, we obtain estimates for the Neumann eigenvalues in terms of the norms of the corresponding extension operators on Sobolev spaces.

\section{Sobolev spaces on metric measure spaces and embedding theorems}

Let $E\subset (X, \rho, \mu)$ be a measurable set. The Lebesgue space $L^p(E)$, $1\leq p<\infty$, is defined (see, for example, \cite{B10}) to be the space of measurable functions $f:E\to\mathbb R$ 
with the finite norm
$$
\|f\|_{L^p(E)}=\left(\int_{E}|f(x)|^p~d\mu(x)\right)^{\frac{1}{p}}.
$$
The space $L^p_{\mathrm{loc}}(E)$ is defined analogously. The following result follows from  \cite[Theorem 4.9]{B10}.

\begin{thm}\label{Brezisthm}
Let $\Omega\subset X = (X, \rho, \mu)$ be a domain. Suppose that $\{f_n\}_{n\in\mathbb{N}}$ is a sequence in $L^q(\Omega)$, $1<q<\infty$, and let $f\in L^q(\Omega)$ such that $\|f_n-f\|_{L^q(\Omega)}\to 0$ as $n\to\infty$. Then there exists a subsequence $\{f_{n_k}\}_{n_k\in\mathbb{N}}$ and a function $g\in L^q(\Omega)$ such that $f_{n_k}\to f$ a.e. in $\Omega$ and $|f_{n_k}|\leq g(x)$ a.e. in $\Omega$ for all $n_k$.
\end{thm}

Suppose that $\mu$ is a doubling measure on the metric measure space $X$ that is there exists a (minimal) constant $C_{\mu}>0$ such that the inequality
$$
 \mu(B(x,2r)) \leq C_{\mu} \mu(B(x,r))
$$
holds for all $x \in X$ and $r>0$. Then $X = (X, \rho, \mu)$ is called a doubling metric measure space.

Recall the notion of Sobolev spaces on doubling metric measure spaces \cite{HK00,Sh00}. The Newtonian-Sobolev space $N^{1,p}(X)$, $1<p<\infty$, is defined to be the space of measurable functions $f : X \to \mathbb{R}$ such that $f \in L^p(X)$ and there exists a Borel function $g: X \to [0, \infty]$ such that $g \in L^p(X)$ and 
    \begin{equation}\label{Newtoneq}
        |f(\gamma(a))-f(\gamma(b))| \leq \int_\gamma g \, ds
    \end{equation}
for $p$-a.e. rectifiable curve $\gamma: [a,b] \to X$, where the integral on $\gamma$ denotes the line integral of $g$ along $\gamma$ and $p$-a.e. means that the property holds for all curves except a family curves with a $p$-modulus zero \cite{MRSY09}. 

The function $g$ which satisfies \eqref{Newtoneq} for $p$-a.e. rectifiable curve $\gamma: [a,b] \to X$ is called a \textit{$p$-weak upper gradient of $f$} and is denoted as $\nabla_{up} f$. If a weak upper gradient $g$ satisfies \eqref{Newtoneq} for every rectifiable curves $\gamma: [a,b] \to X$, then it is called an \textit{upper gradient of $f$}. By Mazur's lemma and by Fuglede's lemma \cite{Sh00}, if $1<p<\infty$, then there exists only one \textit{minimal $p$-weak upper gradient $\nabla_{min} f \in L^p(X)$}, defined up to a set of measure zero.

The Newtonian-Sobolev space $N^{1,p}(X)$ is a Banach space with the norm
    $$
        \|f\|_{N^{(1,p)}X} = \|f\|_{L^p(X)} + \inf\limits_{\nabla_{up} f}\|\nabla_{up} f\|_{L^p(X)},
    $$
where the infimum is taken over all $p$-weak upper gradients (equivalently, over all upper gradients). If $1 < p < \infty$, we have
    $$
        \|f\|_{N^{1,p}(X)} = \|f\|_{L^p(X)} + \|\nabla_{min} f\|_{L^p(X)}.
    $$

Let $X$ be a doubling metric measure space. Then every function $f$, that belongs to $N^{1,p}(X)$, has a representative that is well-defined up to a set of a Sobolev $p$-capacity zero \cite{BB,GT02,KM96}.

The metric measure space $X = (X, \rho, \mu)$ is said to satisfy a \textit{relative lower volume decay} of order $\nu > 0$ if there is a constant $C \geq 1$ such that
$$
	\frac{\mu(B)}{\mu(B_0)} \geq C \Big(\frac{r}{r_0}\Big)^\nu
$$
whenever $B_0$ is an arbitrary ball of radius $r_0$ and $B=B(x,r)$, $x \in B_0$, $r \leq r_0$. It was shown in \cite[Lemma 8.1.13]{HKST}, that doubling metric measure spaces support this property with
$$
    \nu = \log_2 C_\mu.
$$

The metric measure space $X = (X, \rho, \mu)$ supports the \textit{weak $p$-Poincar\'e inequality} \cite{HK00}, if for all balls $B \subset X$, for all functions $f \in L^1(\sigma B)$, $\sigma>1$, and for all $p$-weak upper gradients $\nabla_{up} f$ of $f$ in the ball $\sigma B$ the following inequality holds:
    \begin{equation}\label{Poincareineq}
        \frac{1}{\mu(B)} \int_B |f - f_B| \, d\mu(x) \leq C_p (\diam(B))\left(\frac{1}{\mu(\sigma B)} \int_{\sigma B} (\nabla_{up} f)^p \, d\mu(x)\right)^{\frac{1}{p}}
    \end{equation}
with the constant $C_p$ that does not depend on $B$, $f$, and $\nabla_{up} f$. Doubling metric measure spaces $X = (X, \rho, \mu)$, which supports the weak $p$-Poincar\'e inequality, also supports the standard Sobolev embeddings on balls.

Let $\Omega \subset X$ be a bounded domain. We consider $\Omega$ as a doubling metric measure space with the metric $\rho$ and measure $\mu$ induced from $X$. The \textit{seminormed Newtonian-Sobolev space} $S^{1,p}(\Omega)$, for $1 < p < \infty$,  is defined \cite{MU24} as a space of measurable functions $f : \Omega \to \mathbb{R}$ such that $f \in L^1_{\loc}(\Omega)$, and there exists a minimal $p$-weak upper gradient $\nabla_{min} f \in L^p(\Omega)$, with the finite seminorm:
    $$
        \|f\|_{S^{1,p}(\Omega)} = \|\nabla_{min} f\|_{L^p(\Omega)}.
    $$

It has been proved \cite[Theorem 9.1.15]{HKST} that a doubling metric measure space $X = (X, \rho, \mu)$, which supports a weak $p$-Poincar\'e inequality, also supports the standard Sobolev embeddings on balls. Consequently,
$$
N^{1,p}(B) = S^{1,p}(B), \,\,1<p<\infty,
$$
for any ball $B$ in $X = (X, \rho, \mu)$.

The following theorem has been proved in \cite[Theorem 9.1.15]{HKST}: 

\begin{thm}
\label{embedding_1}
 Let $B\subset X = (X, \rho, \mu)$ be a ball in a doubling metric measure space $X$ supporting the weak $p$-Poincar\'e inequality. Then the embedding operator
 $$
    i: N^{1,p}(B) \to L^{q}(B), 1<p<\nu,
 $$
 is bounded for $1<q\leq p^{\ast}={\nu p}/{(\nu-p)}$, and is compact for $1<q<p^{\ast}$.
\end{thm}

For the rest of the paper, we assume that the metric space $(X, \rho, \mu)$ is a doubling metric measure space which supports the weak $p$-Poincar\'e inequality. The examples of such spaces are the Euclidean space $\mathbb R^n$ and the Carnot-Carath\'eodory spaces. 

Let us define the Sobolev $(p,p^{\ast})$-embedding domains: Let $\Omega\subset X = (X, \rho, \mu)$ be a domain. Then $\Omega$ is called a Sobolev $(p,p^{\ast})$-embedding domain, if the embedding operator
$$
    i: N^{1,p}(\Omega) \to L^{q}(\Omega), 1<p<\infty,
 $$
 is bounded for every $1<q\leq p^{\ast}$, and is compact for every $1<q<p^{\ast}$.
 

We remark that Theorem~\ref{embedding_2} give examples of Sobolev $(p,p^{\ast})$ embedding domains on metric measure spaces. The other approach  to the Sobolev $(p,p^{\ast})$-embedding domains is connected with the composition operators theory \cite{MU24}.

\section{Variational eigenvalues on metric measure spaces}

Let $\Omega\subset X = (X, \rho, \mu)$ be a domain. We define the first non-trivial variational Neumann $(p,q)$-eigenvalue, $1<q,p<\infty$, as 
$$
\lambda_{p,q}(\Omega)\\=\inf \left\{\frac{\|\nabla_{min} u\|_{L^p(\Omega)}^p}{\|u\|_{L^q(\Omega)}^p} :
u \in N^{1,p}(\Omega) \setminus \{0\}, \int_{\Omega} |u|^{q-2}u~d\mu(x)=0 \right\}.
$$

On the first step we prove the following auxiliary lemma.

\begin{lem}\label{lem1}
Let $\Omega\subset X = (X, \rho, \mu)$ be a Sobolev $(p,p^{\ast})$-embedding domain, $1<p<\infty$. Let $1<q<p^{\ast}$ and suppose that $v\in N^{1,p}(X)\setminus\{0\}$ be such that $\int_{\Omega_{\gamma}}|v|^{q-2}v\,d\mu(x)=0$. Then there exists a constant $C=C(\Omega)>0$ such that
$$
\|v\|_{L^p(\Omega)}\leq C\|\nabla_{min} v\|_{L^p(\Omega)}.
$$
\end{lem}

\begin{proof}
By contradiction, suppose for every $n\in\mathbb{N}$, there exists $v_n\in N^{1,p}(\Omega)\setminus\{0\}$ such that $\int_{\Omega}|v_n|^{q-2}v_n\,d\mu(x)=0$ and
\begin{equation}\label{cn}
\|v_n\|_{L^p(\Omega\
)}>n\|\nabla v_n\|_{L^p(\Omega)}.
\end{equation}
Without loss of generality, let us assume that $\|v_n\|_{L^p(\Omega)}=1$. If not, we define
$$
u_n=\frac{v_n}{\|v_n\|_{L^p(\Omega)}},
$$
then $\|u_n\|_{L^p(\Omega)}=1$ and \eqref{cn} holds for $u_n$ and also $\int_{\Omega}|u_n|^{q-2}u_n\,d\mu(x)=0$. By \eqref{cn}, since $\|\nabla_{min} v_n\|_{L^p(\Omega)}\to 0$ as $n\to\infty$, we have $\{v_n\}_{n\in\mathbb{N}}$ is uniformly bounded in $N^{1,p}(\Omega)$.

Since $\Omega$ is a $(p,p^{\ast})$-Sobolev embedding domain, there exists $v\in N^{1,p}(\Omega)$ such that
$$
v_n{\rightharpoonup} v\text{ weakly\,in }N^{1,p}(\Omega),\quad v_n\to v\text{ strongly\,in }L^q(\Omega),\quad\forall \,1<q<p^{*}
$$
and
$$
\nabla_{min} v_n{\rightharpoonup}\nabla_{min} v\text{ weakly\,in } L^p(\Omega).
$$
Hence, by Theorem \ref{Brezisthm}, there exists a function $g\in L^q(\Omega)$ such that
$$
|v_n|\leq g\text{ a.e.\, in }\Omega.
$$
Since $\|\nabla_{min} v_n\|_{L^p(\Omega)}\to 0$ as $n\to\infty$, we have $\nabla_{min} v_n{\rightharpoonup} 0$ weakly in $L^p(\Omega)$, hence $\nabla v=0$ a.e. in $\Omega$, which gives that $v=\text{constant}$ a.e. in $\Omega$. This combined with the fact that
$$
0=\lim_{n\to\infty}\int_{\Omega}|v_n|^{q-2}v_n\,d\mu(x)=\int_{\Omega}|v|^{q-2}v\,d\mu(x)
$$
gives that $v=0$ a.e. in $\Omega$. This contradicts the hypothesis that $\|v_n\|_{L^p(\Omega)}=1$. This completes the proof.
\end{proof}

Now we can prove the Min-Max Principle and existence of minimizers.

\begin{thm}\label{minthm}
Let $\Omega\subset X = (X, \rho, \mu)$ be a Sobolev $(p,p^{\ast})$-embedding domain, $1<p<\infty$. Then for $1<q<p^{\ast}$, there exists $u\in N^{1,p}(\Omega)\setminus\{0\}$ such that $\int_{\Omega}|u|^{q-2}u\,d\mu(x)=0$. Moreover,
\begin{multline*}
\lambda_{p,q}(\Omega)=\inf \left\{\frac{\|\nabla_{min} u\|_{L^p(\Omega)}^p}{\|u\|_{L^q(\Omega)}^p} : u \in N^{1,p}(\Omega) \setminus \{0\},
\int_{\Omega} |u|^{q-2}u~d\mu(x)=0 \right\}\\
=\frac{\|\nabla_{min} u\|_{L^p(\Omega)}^p}{\|u\|_{L^q(\Omega)}^p}.
\end{multline*}
Further, $\lambda_{p,q}(\Omega)>0.$
\end{thm}

\begin{proof}
Let $n\in\mathbb{N}$ and define the functionals $G: N^{1,p}(\Omega)\to\mathbb{R}$ by
$$
G(v)=\int_{\Omega}|v|^{q-2}v\,d\mu(x),
$$
and $H_\frac{1}{n}: N^{1,p}(X)\to\mathbb{R}$ by
$$
H_\frac{1}{n}(v)=\|\nabla_{min} v\|_{L^p(\Omega)}^p-(\lambda_{p,q}+\frac{1}{n})\|v\|_{L^q(\Omega)}^p,
$$
where we denoted $\lambda_{p,q}(\Omega)$ by $\lambda_{p,q}$.
By the definition of infimum, for every $n\in\mathbb{N}$, there exists a function $u_n\in N^{1,p}(\Omega)\setminus\{0\}$ such that
$$
\int_{\Omega}|u_n|^{q-2}u_n\,d\mu(x)=0\,\,\text{ and}\,\, H_\frac{1}{n}(u_n)<0.
$$
Without loss of generality, let us assume that $\|\nabla_{min} u_n\|_{L^p(\Omega)}=1$. By Lemma \ref{lem1}, the sequence $\{u_n\}_{n\in\mathbb{N}}$ is uniformly bounded in $N^{1,p}(\Omega)$.
Since $\Omega$ is a Sobolev $(p,p^\ast)$-embedding domain, the embedding operator
$$
N^{1,p}(\Omega)\hookrightarrow L^q(\Omega),\quad 1<q<p^{*},
$$
is compact, hence there exists $u\in N^{1,p}(\Omega)$ such that $u_n{\rightharpoonup} u$ weakly in $N^{1,p}(\Omega)$, $u_n\to u$ strongly in $L^q(\Omega)$ and $\nabla_{min} u_n{\rightharpoonup} \nabla_{min} u$ weakly in $L^p(\Omega)$.

Hence, by Theorem \ref{Brezisthm}, there exists a function $g\in L^q(X)$ such that
$$
|u_n|\leq g\,\,\text{ a.e. in}\,\,\Omega.
$$
Since $|u_n|\leq g$ a.e. in $\Omega$ and $u_n\to u$ a.e. in $\Omega$, then
$$
||u_n|^{q-2}u_n|\leq |u_n|^{q-1}\leq |g|^{q-1}\in L^{q'}(\Omega).
$$
So, by the Lebesgue Dominated Convergence Theorem (see, for example, \cite{Fe69}), it follows that
$$
0=\lim_{n\to\infty}\int_{\Omega}|u_n|^{q-2}u_n\,d\mu(x)=\int_{\Omega}|u|^{q-2}u\,d\mu(x).
$$
So,
$$
\int_{\Omega}|u|^{q-2}u\,d\mu(x)=0.
$$
Due to $H_\frac{1}{n}(u_n)<0$, we obtain
\begin{equation}\label{1}
\|\nabla_{min} u_n\|_{L^p(\Omega)}^p-(\lambda_{p,q}+\frac{1}{n})\|u_n\|_{L^q(\Omega)}^p<0.
\end{equation}
Since $\nabla_{min} u_n{\rightharpoonup} \nabla_{min} u$ weakly in $L^p(\Omega)$, by the weak lower semicontinuity of norm, we have
$$
\|\nabla_{min} u\|^p_{L^p(\Omega)}\leq \lim\inf_{n\to\infty}\|\nabla_{min} u_n\|^p_{L^p(\Omega)}=1.
$$
So, by passing to the limit in \eqref{1}, we get
$$
\lambda_{p,q}\geq \frac{\|\nabla u\|^p_{L^p(\Omega)}}{\|u\|^p_{L^q(\Omega)}}.
$$
Therefore, by the definition of $\lambda_{p,q}$, we obtain
$$
\lambda_{p,q} = \frac{\|\nabla u\|^p_{L^p(\Omega)}}{\|u\|^p_{L^q(\Omega)}}.
$$
Now, since $H_{\frac{1}{n}}(u_n)<0$ and $\|\nabla_{min} u_n\|^p_{L^p(\Omega)}=1$, we have
$$
1-(\lambda_{p,q}+\frac{1}{n})\|u_n\|_{L^q(\Omega)}^p<0.
$$
Letting $n\to\infty$, we get
$$
\|u\|^p_{L^q(\Omega)}\lambda_{p,q}\geq 1,
$$
which gives $\lambda_{p,q}>0$ and $u\neq 0$ a.e. in $\Omega$.
\end{proof}

We denote by $D(f)$ the class of all non-negative measurable functions $g$ in $\Omega$ such that
$$
|f(x)-f(y)|\leq\rho(x,y)(g(x)+g(y))\,\,\text{$\mu$-a.e. in}\,\,\Omega.
$$
The elements of $D(f)$ are called as generalized gradients of $f$.

The Hajlasz-Sobolev space $M^{1,p}(\Omega)$, $1<p<\infty$, consists of all $f\in L^p(\Omega)$ such that $D(f)\cap L^p(\Omega)\ne \emptyset$. The norm of $f\in M^{1,p}(\Omega)$ is defined by 
    $$
        \|f\|_{M^{(1,p)}(\Omega)} = \|f\|_{L^p(\Omega)} + \inf\limits_{g\in D(f)\cap L^p(\Omega)}\|g\|_{L^p(\Omega)},
    $$
Recall the following result \cite{HKT}:
\begin{thm}
\label{HKT}
Let $\Omega\subset X = (X, \rho, \mu)$ be a doubling domain  which supports the weak $p$-Poincar\'e inequality. Then $M^{1,p}(\Omega)=N^{1,p}(\Omega)$.
\end{thm} 

Let $\Omega\subset X = (X, \rho, \mu)$ be a domain. The $\Omega$ satisfies the measure density condition \cite{HKT}, if there exists a constant $c_0>0$ such that 
$$
\mu(\Omega\cap B(x,r))\geq c_0 \mu(B(x,r))
$$
for all $x\in\overline{\Omega}$ and for all $0<r\leq 1$.

By using Theorem~\ref{HKT} we have the following extension theorem \cite{HKT}:
\begin{thm}
\label{ext}
Let $\Omega\subset X = (X, \rho, \mu)$ be a doubling domain  which supports the weak $p$-Poincar\'e inequality. Suppose that $\Omega$ satisfies the measure density condition, Then the extension operator
$$
E: N^{1,p}(\Omega)\to N^{1,p}(X), \,\,1<p<\infty,
$$
is bounded.
\end{thm}

Hence, by combination of Theorem~\ref{embedding_1} and Theorem~\ref{ext} we obtain:

\begin{thm}
\label{embedding_2}
Let $\Omega\subset X = (X, \rho, \mu)$ be a doubling domain  which supports the weak $p$-Poincar\'e inequality. Suppose that $\Omega$ satisfies the measure density condition. Then the embedding operator
$$
    i: N^{1,p}(\Omega) \to L^{q}(\Omega), 1<p<\nu,
 $$
 is bounded for $1<q\leq p^{\ast}={\nu p}/{(\nu-p)}$, and is compact for $1<q<p^{\ast}$.
\end{thm}

Hence, if $\Omega\subset X = (X, \rho, \mu)$ be a doubling domain  which supports the weak $p$-Poincar\'e inequality, $1<p<\nu$ and supporting the measure density condition, then $\Omega$ is a Sobolev $(p,p^\ast)$-domain. Hence, by Theorem \ref{embedding_2}, we have:

\begin{thm}\label{minthm2}
Let $\Omega\subset X = (X, \rho, \mu)$ be a doubling domain  which supports the weak $p$-Poincar\'e inequality, $1<p<\nu$. Suppose that $\Omega$ satisfies the measure density condition.
Then for $1<q<p^{\ast}$, there exists $u\in N^{1,p}(\Omega)\setminus\{0\}$ such that $\int_{\Omega}|u|^{q-2}u\,dx=0$. Moreover,
\begin{multline*}
\lambda_{p,q}(\Omega)=\inf \left\{\frac{\|\nabla_{min} u\|_{L^p(\Omega)}^p}{\|u\|_{L^q(\Omega)}^p} : u \in N^{1,p}(\Omega) \setminus \{0\},
\int_{\Omega} |u|^{q-2}u~d\mu(x)=0 \right\}\\
=\frac{\|\nabla_{min} u\|_{L^p(\Omega)}^p}{\|u\|_{L^q(\Omega)}^p}.
\end{multline*}
Further, $\lambda_{p,q}(\Omega)>0.$
\end{thm}

\section{Estimates of variational eigenvalues}

Let $\Omega\subset X = (X, \rho, \mu)$ be a Sobolev $(p,p^{\ast})$-embedding domain, for $1<p<\infty$ and $1<q<p^{\ast}=\nu p/(\nu-p)$.
By Theorem \ref{minthm}, the first non-trivial Neumann eigenvalue $\lambda_{p,q}(\Omega)$ can be characterized as
\begin{equation*}
\lambda_{p,q}(\Omega)=\inf \left\{\frac{\|\nabla_{min} u\|_{L^p(\Omega)}^p}{\|u\|_{L^q(\Omega)}^p} : u \in N^{1,p}(\Omega) \setminus \{0\},
\int_{\Omega} |u|^{q-2}u~d\mu(x)=0 \right\}.
\end{equation*}

Furthermore, since $\lambda=0$ is a trivial eigenvalue and $u=c$ are corresponding eigenfunctions, we have
\begin{multline*}
\lambda_{p,q}(\Omega)=\inf\limits_{(u-c) \in N^{1,p}(\Omega)}\sup\limits_{c \in \mathbb R}\frac{\|\nabla_{min} (u-c)\|_{L^p(\Omega)}^p}{\|u-c\|_{L^q(\Omega)}^p}\\
=\inf\limits_{u \in N^{1,p}(\Omega)\setminus \{0\},~ \int_{\Omega} |u|^{q-2}u~d\mu(x)=0}\frac{\|\nabla_{min} u\|_{L^p(\Omega)}^p}{\|u\|_{L^q(\Omega)}^p}.
\end{multline*}
Hence $\lambda_{p,q}(\Omega)^{-\frac{1}{p}}$ is equal to the best constant $C_{p,q}(\Omega)$ in the $(p,q)$-Sobolev-Poincar\'e inequality
\begin{equation}\label{bc1}
\begin{split}
\inf\limits_{c \in \mathbb R}\left(\int\limits_{\Omega} |u(x)-c|^q~d\mu(x)\right)^{\frac{1}{q}}=\left(\int\limits_{\Omega} |u(x)-\widetilde{u}|^q~d\mu(x)\right)^{\frac{1}{q}}
\\
\leq C_{p,q}(\Omega)
\left(\int_{\Omega} |\nabla_{min} u(x)|^p~d\mu(x)\right)^{\frac{1}{p}}, \,\,\, u \in N^{1,p}(\Omega),
\end{split}
\end{equation}
where $\widetilde{u}\in\mathbb R$ is defined by the equality: 
$$
\inf\limits_{c \in \mathbb R}\int_{\Omega} |u(x)-c|^q~d\mu(x)=\int\limits_{\Omega} |u(x)-\widetilde{u}|^q~d\mu(x),
$$
see details in \cite{MU25}.

\vskip 0.2cm

Let us consider extension operators on Sobolev spaces for estimates of variational Neumann eigenvalues. The  extension operators traces back to the classical works of A.~P.~Calderon \cite{C61} and E.~M.~Stein \cite{S70}, which proved that Lipschitz bounded domains are $(p,p)$-extension domains for any $1 \leq p \leq \infty$. P.~Jones \cite{J81} generalized this result to the so-called uniform domains.

The problem of estimates of extension operator norms is crucial for the estimates Neumann eigenvalues of of elliptic operators \cite{GPU20}. In this case, not only the existence of extension operators is needed, but also estimates of their norms. In \cite{T15}, estimates for the norms of extension operators in Lipschitz domains were obtained, following the work of \cite{S70}. In \cite{GU24} estimates for the norms of extension operators were obtained by using the composition operators on Sobolev spaces \cite{VU04,VU05}.

Recall that a bounded domain $\Omega \subset X = (X,\rho, \mu)$ is called a $(p,q)$-extension domain \cite{KUZ22}, for $1 \leq q \leq p \leq \infty$, if there exists a bounded operator
$$
E: S^{1,p}(\Omega) \to S^{1,q}(X), ,\,\ 1 \leq q \leq p \leq \infty,
$$
such that $E(f)\big|_{\Omega} = f$ for any function $f \in S^{1,p}(\Omega)$. This operator is called the bounded $(p,q)$-extension operator. The norm of the bounded $(p,q)$-extension operator is defined as
$$
\|E\| = \sup\limits_{f \in S^{1,p}(\Omega)} \frac{\|E(f)\|_{S^{1,q}(X)}}{\|f\|_{S^{1,p}(\Omega)}}.
$$

\vskip 0.2cm

The following theorem generalizes the result of \cite{GPU20,GU24}, which were given in the Euclidean spaces  $\mathbb R^n$.

\begin{thm} Let $X = (X, \rho, \mu)$  be a doubling metric measure space $X$ supporting the weak $p$-Poincar\'e inequality, $1<p<\nu$. 
Suppose that $\Omega\subset\widetilde{\Omega}\subset X$ are $(p,p)$-extension domains. Then
$$
\lambda_{p,q}(\widetilde{\Omega})\leq \|E_{\Omega,\widetilde{\Omega}}\|^p   \lambda_{p,q}(\Omega),\,\,1<q<p^*=\frac{\nu p}{\nu-p},
$$
where $\|E_{\Omega,\widetilde{\Omega}}\|$ denotes the norm of a bounded extension operator
$$
E_{\Omega,\widetilde{\Omega}}:  S^{1,p}(\Omega)\to S^{1,p}(\widetilde{\Omega}).
$$
\end{thm}

\begin{proof}
Because $\Omega$ and $\widetilde{\Omega}$ are $(p,p)$-extension domains, then there exist bounded extension operators
$$
E:  S^{1,p}(\Omega)\to S^{1,p}(B)\,\,\text{and}\,\,E:  S^{1,p}(\widetilde{\Omega})\to S^{1,p}(B),\,\,B\subset X,
$$
where a ball $B\supset \widetilde{\Omega}\supset \Omega$. Hence, because we have in the ball $B\subset X$ the $p$-Poincar\'e inequality \cite{HKST},  there exist bounded extension operators
$$
E:  N^{1,p}(\Omega)\to N^{1,p}(B)\,\,\text{and}\,\,E:  N^{1,p}(\widetilde{\Omega})\to N^{1,p}(B),\,\,B\subset X,
$$

Therefore taking into account above along with the doubling metric measure space $X$ supporting the weak $p$-Poincar\'e inequality, by Teorem~\ref{embedding_1} there exist compact embedding operators
$$
    i: N^{1,p}(\Omega) \to L^{q}(\Omega)\,\,\text{and}\,\,i: N^{1,p}(\widetilde{\Omega}) \to L^{q}(\widetilde{\Omega}), 1<q<p^*.
$$
So, the first non-trivial variational Neumann eigenvalues $\lambda_{p,q}(\Omega)$ and $\lambda_{p,q}(\widetilde{\Omega})$ are well defined in domains $\Omega$ and $\widetilde{\Omega}$ correspondingly.

There exists a bounded extension operator
$$
E_{\Omega,\widetilde{\Omega}}:  S^{1,p}(\Omega)\to S^{1,p}(\widetilde{\Omega}),
$$
such that
$$
(E_{\Omega,\widetilde{\Omega}}(u))(x)=
\begin{cases}
    u(x),\text{ if }x\in\Omega,\\
    v(x),\text{ if }x\in\widetilde{\Omega}\setminus\Omega,
\end{cases}
$$
where $v:\widetilde{\Omega}\setminus\Omega\to\mathbb{R}$ is the extension of the function $u$. Hence, for every $u\in N^{1,p}(\Omega)$, taking into account the inequality \eqref{bc1}, we have
\begin{multline*}
\|u-u_{\Omega}\|_{L^q(\Omega)}=\inf_{c\in\mathbb{R}}\|u-c\|_{L^q(\Omega)}
=\inf_{c\in\mathbb{R}}\|E_{\Omega,\widetilde{\Omega}}u-c\|_{L^q(\Omega)}\\
\leq \inf_{c\in\mathbb{R}}\|E_{\Omega,\widetilde{\Omega}}u-c\|_{L^q(\widetilde{\Omega})}
\leq C_{p,q}(\widetilde{\Omega})\|\nabla E_{\Omega,\widetilde{\Omega}}\|_{L^p(\widetilde{\Omega})}\\
\leq C_{p,q}(\widetilde{\Omega})\|E_{\Omega,\widetilde{\Omega}}\|\|\nabla u\|_{L^p(\Omega)}.
\end{multline*}

Since $\lambda_{p,q}^{-\frac{1}{p}}(\widetilde{\Omega})=C_{p,q}(\widetilde{\Omega})$, the result follows from the above estimate.
\end{proof}

\vskip 0.2cm

\noindent
{\bf Acknowledgement.} P. Garain thanks IISER Berhampur for the Seed grant: IISERBPR/RD/OO/2024/15, Date:
February 08, 2024.

\vskip 0.3cm

\noindent {\textsf{Prashanta Garain\\
Department of Mathematical Sciences,\\
Indian Institute of Science Education and Research Berhampur\\
Berhampur, Odisha 760010, India}\\
\textsf{e-mail}: pgarain92@gmail.com\\

\vskip 0.1cm

\noindent {\textsf{Alexander Ukhlov\\
Department of Mathematics,\\
Ben-Gurion University of the Negev,\\ P.O.Box 653, Beer Sheva, 8410501, Israel}\\
\textsf{e-mail}: ukhlov@math.bgu.ac.il\\

\end{document}